\documentclass[11pt,a4paper]{article}

\usepackage[math]{cellspace}
\usepackage{psfrag}
\usepackage{amsmath,amssymb,cite}
\usepackage{latexsym}
\usepackage{graphicx}
\usepackage{lscape}
\usepackage{afterpage}
\usepackage[color=green!40]{todonotes}

\DeclareMathOperator*{\argmin}{argmin}

\newcommand{\Id}{\ensuremath{\operatorname{Id}}}

\newcommand{\ds}{\displaystyle}
\newcommand{\nexto}{\kern -0.54em}
\newcommand{\dR}{{\rm {I\ \nexto R}}}
\newcommand{\dN}{{\rm {I\ \nexto N}}}
\newcommand{\dZ}{{\cal Z \kern -0.7em Z}}
\newcommand{\dC}{{\rm\hbox{C \kern-0.8em\raise0.2ex\hbox{\vrule
height5.4pt width0.7pt}}}}
\newcommand{\dQ}{{\rm\hbox{Q \kern-0.85em\raise0.25ex\hbox{\vrule
height5.4pt width0.7pt}}}}
\newcommand{\proofbox}{\hspace{\fill}{$\Box$}}

\newtheorem{proposition}{Proposition}

\newenvironment{proof}{Proof.}{\proofbox}

\begin{document}

\author{Authors}

\author{
Heinz H. Bauschke\footnote{ Department of Mathematics, University of British Columbia, Kelowna, B.C. V1V 1V7, Canada. E-mail: heinz.bauschke@ubc.ca\,.}
\and
Regina S. Burachik\footnote{School of Information Technology and Mathematical Sciences, University of South Australia, Mawson Lakes, S.A. 5095, Australia. E-mail: regina.burachik@unisa.edu.au\,.}
\and
C. Yal{\c c}{\i}n Kaya\footnote{School of Information Technology and Mathematical Sciences, University of South Australia, Mawson Lakes, S.A. 5095, Australia. E-mail: yalcin.kaya@unisa.edu.au\,.}
}

\title{\vspace{-10mm}\bf Constraint Splitting and Projection Methods for Optimal Control of Double Integrator}

\maketitle

\vspace*{-7mm}
\begin{center}
Dedicated to the memory of Jonathan M. Borwein
\end{center}

\begin{abstract} {\noindent\sf  We consider the minimum-energy control of a car, which is modelled as a point mass sliding on the ground in a fixed direction, and so it can be mathematically described as the double integrator.  The control variable, representing the acceleration or the deceleration, is constrained by simple bounds from above and below.  Despite the simplicity of the problem, it is not possible to find an analytical solution to it because of the constrained control variable.  To find a numerical solution to this problem we apply three different projection-type methods: (i)~Dykstra's algorithm, (ii)~the Douglas--Rachford (DR) method and (iii)~the Arag\'on Artacho--Campoy (AAC) algorithm. To the knowledge of the authors, these kinds of (projection) methods have not previously been applied to continuous-time optimal control problems, which are infinite-dimensional optimization problems.  The problem we study in this article is posed in infinite-dimensional Hilbert spaces.  Behaviour of the DR and AAC algorithms are explored via numerical experiments with respect to their parameters.  An error analysis is also carried out numerically for a particular instance of the problem for each of the algorithms.
}
\end{abstract}
\begin{verse}
\end{verse}

\pagestyle{myheadings}
\markboth{}{\sf\scriptsize Projection Methods for Optimal Control of Double Integrator\ \ by H. H. Bauschke, R. S. Burachik \& C. Y. Kaya}

\section{Introduction}

In this paper, we provide (to the best of our knowledge also first) application of various best approximation algorithms to solve a continuous-time optimal control problem.  Operator splitting methods were applied previously to discrete-time optimal control problems~\cite{EckFer1998, OdoStaBoy2013}, which are finite-dimensional problems.  In~\cite{OdoStaBoy2013}, for example, the state difference equations comprise the constraint ${\cal A}$, and the box constraints on the state and control variables comprise ${\cal B}$.  The condition of belonging to the sets ${\cal A}$ and ${\cal B}$ are then appended to the objective function via indicator functions.  The original objective function that is considered in~\cite{OdoStaBoy2013} is quadratic in the state and control variables.  In the next step in~\cite{OdoStaBoy2013}, the new objective function is split into its quadratic and convex parts and the Douglas-Rachford splitting method is applied to solve the problem.

In the current paper, we deal with continuous-time optimal control problems, which are infinite-dimensional optimization problems that are set in Hilbert spaces.  After splitting the constraints of the problem, we apply Dykstra's algorithm~\cite{BoyleDykstra}, the Douglas--Rachford (DR) method \cite{DougRach, LM, EckBer, BauCombettes, Svaiter, BauMoursi}, and the Arag\'on Artacho--Campoy (AAC) algorithm~\cite{AAC}, all of which solve the underlying best approximation problem. 

The exposure of the current paper is more in the style of a tutorial.  We pose the problem of minimum-energy control of a simplified model of a car, amounting to the double integrator, where the control variable has simple lower and upper bounds and the initial and terminal state variables are specified.  We split the constraints into two, ${\cal A}$ and ${\cal B}$, representing respectively the state differential equations (the double integrator) along with their boundary conditions and the constraints on the control variable.  We define two subproblems, one subject to ${\cal A}$, and the other one subject to ${\cal B}$.  We take advantage of the relatively simple form of the optimal control problem and derive analytical expressions for the optimality conditions and implement these in defining the projections onto ${\cal A}$ and ${\cal B}$.

The solutions of these subproblems provide the projections of a given point in the control variable space onto the constraint sets ${\cal A}$ and ${\cal B}$, respectively, in some optimal way.  By performing these projections in the way prescribed by the above-listed algorithms, we can ensure convergence to a solution of the original optimal control problem, 

Note that while the minimum-energy control of the double integrator without any constraints on the control variable can be solved analytically, the same problem with (even simple bound, i.e., box) constraints on the control variable can in general be solved only numerically.  This problem should be considered within the framework of control-constrained linear-quadratic optimal control problems for which new numerical methods are constantly being developed---see for example \cite{BurKayMaj2014, AltKaySch2016} and the references therein.

The current paper is a prototype for future applications of projection methods to solving more general optimal control problems.  Indeed, the minimum-energy control of double integrator is a special case of linear quadratic optimal control problems; so, with the reporting of the current study, an extension to more general problems will be imminent.  

The paper is organized as follows.  In Section 2, we state the control-constrained minimum-energy  problem  for the double integrator, and write down the optimality conditions.  We provide the analytical solution for the unconstrained problem.  For the control-constrained case, we briefly describe the standard numerical approach and consider an instance of the problem which we use in the numerical experiments in the rest of the paper.  We define the constraint sets ${\cal A}$ and ${\cal B}$.  In Section~3, we provide the expressions for the projections onto ${\cal A}$ and ${\cal B}$.  We describe the algorithms in Section~4 and in the beginning of Section~5.  In the remaining part of Section~5, we present numerical experiments to study parametric behaviour of the algorithms as well as the errors in the state and control variables with each algorithm.  In Section~6, we provide concluding remarks and list some open problems.

\section{Minimum-Energy Control of Double Integrator}

We consider the minimum-energy control of a car, with a constrained control variable.  Consider the car as a point unit mass, moving on a frictionless ground in a fixed line of action.  Let the position of the car at time $t$ be given by $y(t)$ and the velocity by $\dot{y}(t):=(dy/dt)(t)$.  By Newton's second law of motion, $\ddot{y}(t) = u(t)$, where $u(t)$ is the summation of all the external forces applied on the car, in this case the force simply representing the acceleration and deceleration of the car.  This differential equation model is referred to as the {\em double integrator} in system theory literature, since $y(t)$ can be obtained by integrating $u(t)$ twice.

\noindent 
{\bf Optimal Control Problem.}\ \ Suppose that the total force on the car, i.e., the acceleration or deceleration of the car, is constrained by a magnitude of $a>0$.  Let $x_1 := y$ and $x_2 := \dot{y}$.  Then the problem of minimizing the energy of the car, which starts at a position $x_1(0) = s_0$ with a velocity $x_2(0) = v_0$ and finishes at some other position $x_1(1) = s_f$ with velocity $x_2(1) = v_f$, within one unit of time, can be posed as follows.
\[
\mbox{(P) }\left\{\begin{array}{rl}
\ds\min & \ \ \ds\frac{1}{2}\int_0^1 u^2(t)\,dt  \\[5mm] 
\mbox{subject to} & \ \ \dot{x}_1(t) = x_2(t)\,,\ \ x_1(0) = s_0\,,\ \ x_1(1) = s_f\,, \\[2mm]
& \ \ \dot{x}_2(t) = u(t)\,,\ \ \ \,x_2(0) = v_0\,,\ \ x_2(1) = v_f\,, \ \ \ |u(t)|\le a\,.
\end{array} \right.
\]
Here, the functions $x_1$ and $x_2$ are referred to as the {\em state variables} and $u$ the {\em control variable}. As a first step in writing the conditions of optimality for this optimization problem, define the Hamiltonian function $H$ for Problem (P) simply as
\begin{equation}  \label{Hamiltonian}
H(x_1,x_2,u,\lambda_1,\lambda_2) := \frac{1}{2}\,u^2 + \lambda_1\,x_2 + \lambda_2\,u\,,
\end{equation}
where $\lambda(t) := (\lambda_1(t),\lambda_2(t))\in\dR^2$ is the {\em adjoint variable} (or {\em costate}) {\em vector} such that (see~\cite{Hestenes66})
\begin{equation} \label{adjoint} 
\dot{\lambda}_1 = -\partial H /\partial x_1\quad\mbox{and}\quad
\dot{\lambda}_2 = -\partial H /\partial x_2\,.
\end{equation} 
Equations in~\eqref{adjoint} simply reduce to 
\begin{equation} \label{adjoint_sol} 
\lambda_1(t) = c_1\quad\mbox{and}\quad \lambda_2(t) = -c_1\,t - c_2\,, 
\end{equation} 
where $c_1$ and $c_2$ are real constants.  Let the state variable vector $x(t) := (x_1(t),x_2(t))\in\dR^2$.

\noindent 
{\bf Maximum Principle.}\ \ If $u$ is an optimal control for Problem~(P), then there exists a continuously differentiable vector of adjoint variables $\lambda$, as defined in \eqref{adjoint}, such that $\lambda(t) \neq 0$ for all $t\in[0,t_f]$, and
that, for a.e. $t\in[0,t_f]$,
\begin{equation}  \label{optcont1}
u(t) = \argmin_{v\in[-a,a]} H(x, v, \lambda(t))\,,
\end{equation}
i.e.,
\begin{equation}  \label{optcont2}
u(t) = \argmin_{v\in[-a,a]}\ \frac{1}{2}\,v^2 + \lambda_2(t)\,v\,;
\end{equation}
see e.g. \cite{Hestenes66}.  Condition~\eqref{optcont2} implies that the optimal control is given by
\begin{equation}  \label{optcont3}
u(t) = \left\{\begin{array}{rl}
-\lambda_2(t)\,, &\ \ \mbox{if\ \ } -a\le\lambda_2(t)\le a\,, \\[1mm]
a\,, &\ \ \mbox{if\ \ } \lambda_2(t)\le -a\,, \\[1mm]
-a\,, &\ \ \mbox{if\ \ } \lambda_2(t)\ge a\,.
\end{array} \right.
\end{equation}
From \eqref{optcont3}, we can also conclude that the optimal control $u$ for Problem~(P) is continuous.

When $a$ is large enough, the control constraint does not become active, so the optimal control is simply $-\lambda_2$, and it is a straightforward classroom exercise to find the analytical solution as 
\begin{eqnarray*}
u(t) &=& c_1\,t + c_2\,, \\[1mm]
x_1(t) &=& \frac{1}{6}\,c_1\,t^3 + \frac{1}{2}\,c_2\,t^2 + v_0\,t + s_0\,, \\[1mm]
x_2(t) &=& \frac{1}{2}\,c_1\,t^2 + c_2\,t + v_0\,,
\end{eqnarray*} 
for all $t\in[0,1]$, where
\begin{eqnarray*}
c_1 &=& -12\,(s_f - s_0) + 6\,(v_0 + v_f)\,, \\[1mm]
c_2 &=& 6\,(s_f - s_0) - 2\,(2\,v_0 + v_f)\,.
\end{eqnarray*}
The solution of an instance of Problem~(P), with $s_0 = 0$, $s_f = 0$, $v_0 = 1$, $v_f = 0$, and large $a$, say $a=9$, is depicted in Figure~\ref{fig:unconstr_soln}.  Note that, for all $t\in[0,1]$, $\lambda_2(t) = -u(t) = -6\,t + 4$ and $\lambda_1(t) = c_1 = 6$.  The graphs of $\lambda_1$ and $\lambda_2$ are not displayed for this particular instance.

\begin{figure}[t]
\begin{minipage}{80mm}
\begin{center}
\includegraphics[width=80mm]{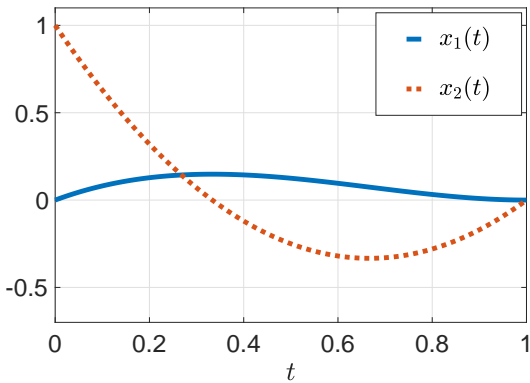} \\[3mm]
(a) Optimal state variables. \\
\end{center}
\end{minipage}
\begin{minipage}{80mm}
\begin{center}
\includegraphics[width=80mm]{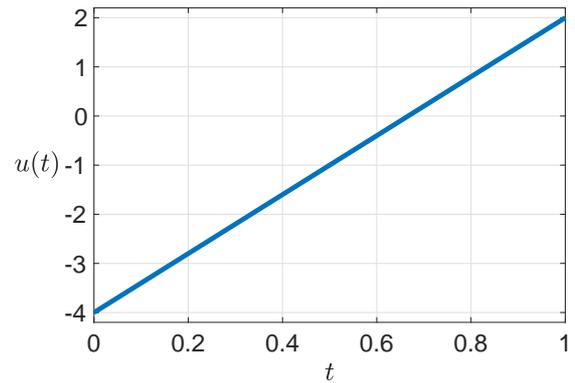} \\[3mm]
(b) Optimal control variable.
\end{center}
\end{minipage}
\caption{\sf Solution of Problem~(P) with large $a$ (so that $u(t)$ is unconstrained), $s_0 = 0$, $s_f = 0$, $v_0 = 1$, $v_f = 0$.} 
\label{fig:unconstr_soln}
\end{figure}

When $a$ is not so large, say $a=2.5$, as we will consider next so that the control constraint becomes active, it is usually not possible to find an analytical solution, i.e., a solution has to be found numerically, as described below.
\begin{figure}[t]
\begin{minipage}{80mm}
\begin{center}
\includegraphics[width=80mm]{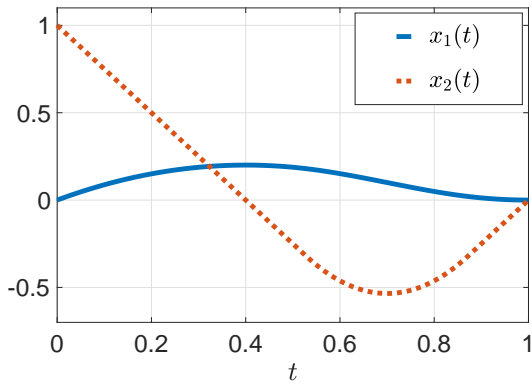} \\[3mm]
(a) Optimal state variables. \\
\end{center}
\end{minipage}
\begin{minipage}{80mm}
\begin{center}
\includegraphics[width=80mm]{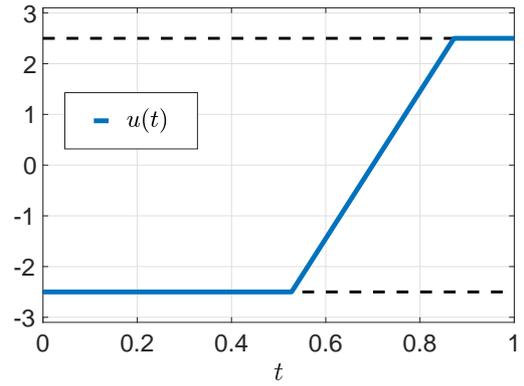} \\[3mm]
(b) Optimal control variable.
\end{center}
\end{minipage}
\\[5mm]
\begin{minipage}{80mm}
\
\end{minipage}
\begin{minipage}{80mm}
\begin{center}
\includegraphics[width=80mm]{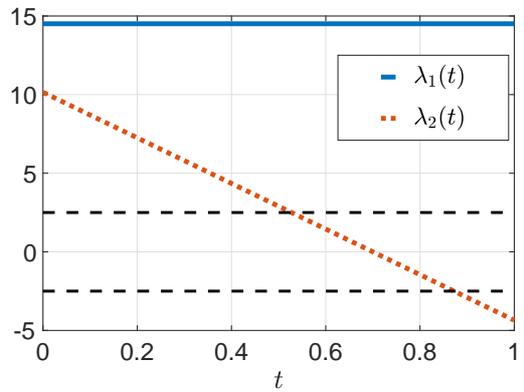} \\[3mm]
(c) Adjoint variables.
\end{center}
\end{minipage}
\caption{\sf Solution of direct discretization of Problem~(P), with $a=2.5$, $s_0 = 0$, $s_f = 0$, $v_0 = 1$, $v_f = 0$.} 
\label{fig:discrete_soln}
\end{figure}

\noindent 
{\bf Numerical Approach.}\ \ A straightforward and popular numerical approach to solving Problem~(P) is to discretize Problem~(P) over a partition of the time horizon $[0,1]$ and then use some finite-dimensional optimization software to get a {\em discrete} (finite-dimensional) {\em solution} for the state and control variables $x(t)$ and $u(t)$.  The discrete solution is an approximation of the continuous-time solution.  This approach is often referred to as the {\em direct method} or the ({\em first-}){\em discretize-then-optimize} approach.  A survey and discussion of Euler discretization of linear-quadratic optimal control problems and convergence of their discretized solutions to their continuous-time solutions can be found in~\cite[Section~5]{BurKayMaj2014}.  

Figure~\ref{fig:discrete_soln} depicts the discrete solution of Problem~(P) with the instance where $a=2.5$, $s_0 = 0$, $s_f = 0$, $v_0 = 1$, $v_f = 0$.   The solution was obtained by pairing up the optimization modelling language AMPL~\cite{AMPL} and the finite-dimensional optimization software Ipopt\cite{WacBie2006}.  The number of discretization nodes was taken to be 2000.  The multipliers of the (Euler approximation of the) state differential equation constraints are provided by Ipopt when it finds an optimal solution to the discretized (finite-dimensional) problem.  These multipliers have been plotted in Figure~\ref{fig:discrete_soln}(c).  It should be noted that the graph of the adjoint variable $\lambda_2(t)$ given in Figure~\ref{fig:discrete_soln}(c) verifies the graph of the optimal control $u(t)$ in Figure~\ref{fig:discrete_soln}(b) via the optimal control rule in~\eqref{optcont3}.  In Figures~\ref{fig:discrete_soln}(b) and (c), the bounds $\pm\,2.5$ have been marked by horizontal dashed lines for ease of viewing.

\noindent 
{\bf Function Spaces.}\ \ For the numerical methods, we consider projection/reflection methods in Hilbert spaces.  The spaces associated with Problem~(P) are set up as follows.  Let $q\in \dN$ and $L^2(0,1;\dR^q)$ be the Banach space of Lebesgue measurable functions 
\[\begin{array}{rccl}
z:&[0,1]&\to&\dR^q\\
&t&\mapsto&(z_1(t),\ldots,z_q(t))^t,\\
\end{array}\]
with finite $L^2$ norm. Namely, define
\[
\|z\|_2 := \left(\sum_{i=1}^q \|z_i\|_2^2\right)^{1/2}
\]
where
\[
\|z_i\|_2 := \left(\int_{0}^{1} \mid\!\! z_i(t)\!\!\mid^2\, \mathrm{d}t\right)^{1/2}\,,
\]
for $i = 1,\ldots,q$, with $\left|\,\cdot\,\right|$ the modulus or absolute value.  In other words,
\[
L^2(0,1;\dR^q) := \left\{z:[0,1]\to \dR^q\::\: \|z\|_2 < \infty\right\}\,.
\]
Furthermore, $W^{1,2}(0,1;\dR^q)$ is the Sobolev space of absolutely continuous functions, namely
\[
W^{1,2}(0,1;\dR^q) = \{z\in L^2(0,1;\dR^q)\left.\right| \dot{z}=dz/dt \in L^2(0,1;\dR^q)\}\,,
\]
endowed with the norm 
\[
\|z\|_{W^{1,2}} := \left(\sum_{i=1}^q \left[\|z_i\|^2_2+ \|\dot{z_i}\|^2_2 \right]\right)^{1/2}\,.
\]
In Problem~(P), the state variable $x\in W^{1,2}(0,1;\dR^2)$ and the control variable $u\in L^2(0,1;\dR)$.

\noindent 
{\bf Constraint Splitting.}\ \ Next, we split the constraints of Problem~(P) into two subsets, ${\cal A}$ and ${\cal B}$. The subset ${\cal A}$ collects together all the feasible control functions satisfying only the dynamics of the car.  The subset ${\cal B}$, on the other hand, collects all the control functions whose values are constrained by $-a$ and $a$. 
\begin{eqnarray} 
{\cal A} := \big\{u\in L^2(0,1;\dR)\ |&& \exists x\in W^{1,2}(0,1;\dR^2)\mbox{ which solves } \nonumber \\
&&\dot{x}_1(t) = x_2(t)\,,\ x_1(0) = s_0\,,\ x_1(1) = s_f\,, \nonumber \\[1mm]
&&\dot{x}_2(t) = u(t)\,,\ \ \,x_2(0) = v_0\,,\ x_2(1) = v_f\,, \ \forall t\in[0,1]\big\}\,, \label{A} \\[2mm]
{\cal B} := \big\{u\in L^2(0,1;\dR)\ |&& -a\le u(t)\le a\,,\ \mbox{for all } t\in[0,1]\big\}\,. \label{B}
\end{eqnarray}
The rationale behind this sort of splitting is as follows:  The problem of minimizing the energy of the car subject to only ${\cal A}$ or only ${\cal B}$ is much easier to solve -- in fact, the solutions can be analytically written in each case.  If, for some given $u$, a solution exists to the two-point boundary-value (TPBVP) in \eqref{A} then that solution is unique by the linearity of the TPBVP~\cite{AscMatRus1995,StoBul2002}.  Note that a control solution $u$ as in \eqref{A} exists by the (Kalman) controllability of the double integrator -- see \cite{Rugh1995}.  So the set ${\cal A}$ is nonempty.  Note that the constraint set~${\cal A}$ is an {\em affine subspace} and ${\cal B}$ a {\em box}.

\section{Projections}

All of the projection methods that we will consider involve projections onto the sets ${\cal A}$ and ${\cal B}$.  The projection onto ${\cal A}$ from a current iterate $u^-$ is the point $u$ solving the following problem. \\
\[
\mbox{(P1) }\left\{\begin{array}{rl}
\ds\min & \ \ \ds\frac{1}{2}\int_0^1 (u(t) - u^-(t))^2\,dt  \\[5mm] 
\mbox{subject to} & \ \ u\in{\cal A}\,.
\end{array} \right.
\]
In (P1), we minimize the squared $L^2$-norm distance between $u^-$ and $u$.  The projection onto ${\cal B}$ from a current iterate $u^-$ is similarly the point $u$ solving the following problem.
\[
\mbox{(P2) }\left\{\begin{array}{rl}
\ds\min & \ \ \ds\frac{1}{2}\int_0^1 (u(t) - u^-(t))^2\,dt  \\[5mm] 
\mbox{subject to} & \ \ u\in{\cal B}\,.
\end{array} \right.
\]

\begin{proposition}[Projection onto \boldmath{${\cal A}$}]  \label{proj_A}
The projection $P_{{\cal A}}$ of $u^-\in L^2(0,1;\dR)$ onto the constraint set ${\cal A}$, as the solution of Problem~{\em(P1)}, is given by
\begin{equation}  \label{u_proj_A}
P_{{\cal A}}(u^-)(t) = u^-(t) + c_1\,t + c_2\,,
\end{equation}
for all $t\in[0,1]$, where
\begin{eqnarray}
&& c_1 = 12\,(x_1(1) - s_f) - 6\,(x_2(1) - v_f)\, \label{c1_sol}\,, \\[1mm]
&& c_2 = -6\,(x_1(1) - s_f) + 2\,(x_2(1) - v_f)\, \label{c2_sol}\,,
\end{eqnarray}
and $x_1(1)$ and $x_2(1)$ are obtained by solving the initial value problem
\begin{eqnarray}  \label{IVP}
&& \dot{x}_1(t) = x_2(t)\,,\ \ \ x_1(0) = s_0\,, \label{IVPa_ref} \\[1mm]
&& \dot{x}_2(t) = u^-(t)\,,\ \ \ x_2(0) = v_0\,, \label{IVPb_ref}
\end{eqnarray}
for all $t\in[0,1]$.
\end{proposition}
\begin{proof}
The Hamiltonian function for Problem~(P1) is
\[
H_1(x_1,x_2,u,\lambda_1,\lambda_2,t) := \frac{1}{2}\,(u - u^-)^2 + \lambda_1\,x_2 + \lambda_2\,u\,,
\]
where the adjoint variables $\lambda_1$ and $\lambda_2$ are defined as in ~\eqref{adjoint}, with $H$ replaced by $H_1$, and the subsequent solutions are given as in~\eqref{adjoint_sol}.  The optimality condition for Problem~(P1) is akin to that in~\eqref{optcont1} for Problem~(P) and, owing to the fact that the control $u$ is now unconstrained, can more simply be written as
\[
\frac{\partial H_1}{\partial u}(x, u, \lambda,t) = 0\,,
\]
which yields the optimal control as $u(t) = u^-(t) - \lambda_2(t)$, i.e.
\begin{equation}\label{u1}
u(t) = u^-(t) + c_1\,t + c_2\,,
\end{equation}
for all $t\in[0,1]$. We need to show that $c_1$ and $c_2$ are found as in \eqref{c1_sol}--\eqref{c2_sol}. Using \eqref{u1} in \eqref{A} yields the following time-varying, linear two-point boundary-value problem.
\begin{eqnarray}  \label{TPBVP}
&& \dot{x}_1(t) = x_2(t)\,,\hspace*{23mm} x_1(0) = s_0\,,\ \ x_1(1) = s_f\,, \label{TPBVPa} \\[1mm]
&& \dot{x}_2(t) = u^-(t) + c_1\,t + c_2\,,\ \ \ \,x_2(0) = v_0\,,\ \ x_2(1) = v_f\,, \label{TPBVPb}
\end{eqnarray}
for all $t\in[0,1]$.  In other words, Problem~(P1) is reduced to solving Equations~\eqref{TPBVPa}--\eqref{TPBVPb} for the unknown parameters $c_1$ and $c_2$.  Once $c_1$ and $c_2$ are found, the projected point $u$ in \eqref{u1} is found.  Since Equations~\eqref{TPBVPa}--\eqref{TPBVPb} are linear in $x_1$ and $x_2$, a {\em simple shooting technique}~\cite{AscMatRus1995,StoBul2002} provides the solution for $c_1$ and $c_2$ in just one iteration.  The essence of this technique is that the initial-value problem (IVP)
\begin{eqnarray}  \label{TPBVPc}
&& \dfrac{\partial {z}_1(t,c)}{\partial t} = z_2(t,c)\,,\hspace*{21mm} z_1(0,c) = s_0\,, \label{TPBVPca} \\[1mm]
&& \dfrac{\partial {z}_2(t,c)}{\partial t} = u^-(t) + c_1\,t + c_2\,,\ \ \ \,z_2(0,c) = v_0\,, \label{TPBVPcb}
\end{eqnarray}
for all $t\in[0,1]$, is solved repeatedly, so as to make the {\em discrepancy at\ $t=1$} vanish. Namely, we seek a parameter $c := (c_1, c_2)$ such that $z_1(1,c)-s_f=0$ and $z_2(1,c)-v_f=0$. The procedure is as follows. For a given $c$, there exists a unique solution $z(t,c):=(z_1(t,c),z_2(t,c))$ of \eqref{TPBVPca}--\eqref{TPBVPcb}.  Define the {\em near-miss} (vector) {\em function} $\varphi:\dR^2\to\dR^2$ as follows:
\begin{equation}\label{fi}
\varphi(c) := \left[\begin{array}{c} z_1(1,c) - s_f \\[2mm] z_2(1,c) - v_f \end{array}\right].
\end{equation}
The Jacobian of the near-miss function is
\[
J_\varphi(c) :=
\left[
\begin{array}{cc}
&\\
\dfrac{\partial z_1(1,c)}{\partial c_1} & \dfrac{\partial z_1(1,c)}{ \partial c_2}\\
&\\
\dfrac{\partial z_2(1,c) }{\partial c_1} & \dfrac{\partial z_2(1,c)}{\partial c_2}\\
&\\
\end{array}
\right]
\]  
The shooting method looks for a pair $c$ such that $\varphi(c) := 0$ (i.e., a pair $c$ such that the final boundary conditions are met). Expanding $\varphi$ about, say, $\overline{c}=0$, and discarding the terms of order $2$ or higher, we obtain
\[
\varphi(c) \approx \varphi(0)+ J_\varphi(0)\,c\,.
\]
Substituting $\varphi(c)=0$ in the above expression, replacing ``$\approx$'' with ``$=$'', and re-arranging, gives the single (Newton) iteration of the shooting method:
\begin{equation}  \label{shooting_itn}
c = -[J_{\varphi}(0)]^{-1}\varphi(0)\,.
\end{equation}
The components $(\partial z_i/\partial c_j)(1,c)$, $i,j = 1,2$, of $J_\varphi(c)$, can be obtained by solving the variational equations for~\eqref{TPBVPa}--\eqref{TPBVPb} with respect to $c_1$ and $c_2$, i.e., by solving the following system for $(\partial z_i/\partial c_j)(\cdot,c)$:
\begin{eqnarray*}
&& \frac{\partial}{\partial t}\!\left(\frac{\partial z_1}{\partial c_1}\right)\!(t,c) = \frac{\partial z_2}{\partial c_1}\!(t,c) \,,\quad \frac{\partial z_1}{\partial c_1}(0,c) = 0\,, \\[1mm]
&& \frac{\partial}{\partial t}\!\left(\frac{\partial z_1}{\partial c_2}\right)\!(t,c) = \frac{\partial z_2}{\partial c_2}\!(t,c) \,,\quad \frac{\partial z_1}{\partial c_2}(0,c) = 0\,, \\[1mm]
&& \frac{\partial}{\partial t}\!\left(\frac{\partial z_2}{\partial c_1}\right)\!(t,c) = t\,,\hspace*{14mm} \frac{\partial z_2}{\partial c_1}(0,c) = 0\,, \\[1mm]
&&\frac{\partial}{\partial t}\left(\frac{\partial z_2}{\partial c_2}\right)\!(t,c) = 1\,,\hspace*{14mm} \frac{\partial z_2}{\partial c_2}(0,c) = 0\,.
\end{eqnarray*}
Elementary calculations lead to the following solution of the above system:
\[
\frac{\partial z}{\partial c}(t,c) =  \left[\begin{array}{cc} t^3/6\ &\ t^2/2 \\[2mm]
t^2/2\ &\ t \end{array}\right]\,,
\]
which is independent of $c$. Hence,
\[
J_\varphi(0) = \frac{\partial z}{\partial c}(1,0) =  \left[\begin{array}{cc} 1/6\ &\ 1/2 \\[2mm]
1/2\ &\ 1 \end{array}\right]\,,
\]
with inverse:
\begin{equation}\label{J}
\left[\frac{\partial z}{\partial c}(1,0)\right]^{-1} = \left[J_\varphi(0)\right]^{-1}=
\left[\begin{array}{cc} -12\ &\ 6 \\[2mm]
6\ &\ -2 \end{array}\right]\,.
\end{equation}
Setting $(x_1(\cdot),x_2(\cdot)):=(z_1(\cdot,0),z_2(\cdot,0))$, the IVP \eqref{TPBVPca}--\eqref{TPBVPcb} becomes \eqref{IVPa_ref}--\eqref{IVPb_ref}. Then substitution of \eqref{fi} and \eqref{J} with $c=0$ into Equation~\eqref{shooting_itn}, and expanding out, yield \eqref{c1_sol}--\eqref{c2_sol}. The proof is complete. \end{proof}

\begin{proposition}[Projection onto \boldmath{${\cal B}$}]  \label{proj_B}
The projection $P_{{\cal B}}$ of $u^-\in L^2(0,1;\dR)$ onto the constraint set ${\cal B}$, as the solution of Problem~{\em(P2)},  is given by
\begin{equation}  \label{u_proj_B}
P_{{\cal B}}(u^-)(t) = \left\{\begin{array}{rl}
u^-(t)\,, &\ \ \mbox{if\ \ } -a\le u^-(t)\le a\,, \\[1mm]
-a\,, &\ \ \mbox{if\ \ } u^-(t)\le -a\,, \\[1mm]
a\,, &\ \ \mbox{if\ \ } u^-(t)\ge a\,,
\end{array} \right.
\end{equation}
for all $t\in[0,1]$. \end{proposition}
\begin{proof}
The expression~\eqref{u_proj_B} is the straightforward solution of Problem~(P2).
\end{proof}

\section{Best Approximation Algorithms}

In this section, we discuss best approximation algorithms.  In the following,
\begin{equation}
\text{$X$ is a real Hilbert space}
\end{equation}
with inner product $\langle\cdot,\cdot\rangle$, induced norm $\|\cdot\|$.  We also assume that
\begin{equation}
\text{
$A$ is a closed affine subspace of $X$, and $B$ is a nonempty closed convex subset of $X$.
}
\end{equation}
Given $z\in X$, our aim is to find
\begin{equation}
P_{A\cap B}(z),
\end{equation}
the projection of $z$ onto the intersection $A\cap B$ which we assume to be \emph{nonempty}.
We also assume that we are able to compute the projectors onto the constraints $P_A$ and $P_B$.

Many algorithms are known which could be employed to find $P_{A\cap B}(z)$; here, however, we focus on three simple methods that do not require a product space set-up as some of those considered, in, e.g., \cite{BauCombettes, BauKoch, Combettes2003, Combettes2009}.

In the next section, we will numerically test these algorithms when $X =  L^2(0,1;\dR)$, $A=\mathcal{A}$,  $B=\mathcal{B}$, and $z=0$.

\subsection{Dykstra's Algorithm}
\label{ss:Dykstra}

We start with Dykstra's algorithm (see \cite{BoyleDykstra}), which operates as follows\footnote{In the general case, there is also an auxiliary sequence $(p_n)$ associated with $A$; however, because $A$ is an affine subspace, it is not needed in our setting.}:  Set $a_0 := z$ and $q_0 := 0$.  Given $a_n,q_n$, where $n\geq 0$, update
\begin{equation}
b_{n} := P_B(a_n+q_n),
\;\;
a_{n+1} := P_A(b_{n}),
\;\;
\text{and}
\;\;
q_{n+1} := a_{n}+q_n-b_{n}.
\end{equation}
It is known that both $(a_n)_{n\in\mathbb{B}}$ and $(b_n)_{n\in\mathbb{N}}$ converge \emph{strongly} to $P_{A\cap B}(z)$.

\subsection{Douglas--Rachford Algorithm}
\label{ss:DR}

Given $\beta>0$, we specialize the Douglas--Rachford algorithm (see
\cite{DougRach}, \cite{LM} and \cite{EckBer}) 
to minimize the sum of the two functions $f(x)=\iota_B(x) +
\tfrac{\beta}{2}\|x-z\|^2$ and $g :=
\iota_A$ which have respective proximal
mappings (see \cite[Proposition~24.8(i)]{BauCombettes})
$P_f(x) = P_B\big(\tfrac{1}{1+\beta}x+\tfrac{\beta}{1+\beta}z\big)$ and
$P_g =P_A$.
Set $\lambda := \tfrac{1}{1+\beta}\in\left]0,1\right[$. 
It follows that the Douglas--Rachford operator 
$ T :=  \Id - P_f + P_g(2P_f-\Id)$
turns into 
\begin{align}
Tx &= x-P_B\big(\lambda x+(1-\lambda)z\big)+P_A\Big(2P_B\big(\lambda
x+(1-\lambda)z\big)-x\Big).
\end{align}
Now let $x_0\in X$ and given $x_n\in X$, where $n\geq 0$, update
\begin{equation}
\label{e:180304a}
b_n:= P_B\big(\lambda x_n+(1-\lambda)z\big),\;\;
x_{n+1} := Tx_n 
= x_n-b_n
+P_A\big(2b_n-x_n\big).
\end{equation}
Then it is known (see \cite{Svaiter} or \cite{BauMoursi}) that
$(b_n)_{n\in\mathbb{N}}$ converges weakly to $P_{A\cap B}(z)$.
Note that \eqref{e:180304a} simplifies to 
\begin{equation}
x_{n+1} := x_n - P_B(\lambda x_n)+P_A\big(2P_B(\lambda x_n)-x_n\big)
\quad\text{provided that $z=0$.}
\end{equation}

\subsection{Arag\'on Artacho--Campoy Algorithm}
\label{ss:AAC}

The Arag\'on Artacho--Campoy (AAC) Algorithm was recently presented in \cite{AAC}. Given two fixed parameters $\alpha$ and $\beta$ in $\left]0,1\right[$, define 
\begin{align}
Tx &= (1-\alpha)x +
\alpha\Bigg(2\beta\bigg(P_A\Big(2\beta\big(P_B(x+z)-z\big)-x+z\Big)-z\bigg)+x+2\beta\big(z-P_B(x+z)\big)\Bigg)\notag\\
&=x +
2\alpha\beta\Bigg(P_A\Big(2\beta\big(P_B(x+z)-z\big)-x+z\Big)-P_B(x+z) \Bigg)\,.
\end{align}
Now let $x_0\in X$ and given $x_n\in X$, where $n\geq 0$, update
\begin{equation}
b_{n} := P_B(x_n+z),
\end{equation}
and 
\begin{equation}
\label{e:180302a}
x_{n+1} := Tx_n = x_n +
2\alpha\beta\bigg(P_A\Big(2\beta\big(b_{n}-z\big)-x_n+z\Big)-b_{n}\bigg).
\end{equation}
By \cite[Theorem~4.1(iii)]{AAC},
the sequence $(b_n)_{n\in\mathbb{N}}$ converges strongly to
$P_{A\cap B}(z)$
provided that\footnote{It appears that this constraint qualification is not easy
to check in our setting.} $z-P_{A\cap B}(z)\in (N_A+N_B)(P_{A\cap B}z)$.
Note that \eqref{e:180302a} simplifies to 
\begin{equation}
x_{n+1} := Tx_n = x_n +
2\alpha\beta\Big(P_A\big(2\beta P_Bx_n-x_n\big)-P_Bx_n\Big)
\quad\text{provided that $z=0$.}
\end{equation}


\section{Numerical Implementation}

\subsection{The algorithms}

In this section, we gather the algorithms considered abstractly and explain how we implemented them.

We start with Dykstra's algorithm from Subsection~\ref{ss:Dykstra}.

\noindent
{\bf Algorithm 1 (Dykstra)}
\begin{description}
\item[Step 1] ({\em Initialization}) Choose the initial iterates
$u^0=0$ and $q^0=0$. 
Choose a small parameter $\varepsilon>0$, and set $k=0$. 
\item[Step 2] ({\em Projection onto ${\cal B}$})  Set $u^- = u^{k} + q^k$. 
Compute $\widetilde{u} = P_{{\cal B}}(u^-)$ by
using \eqref{u_proj_B}.
\item[Step 3] ({\em Projection onto ${\cal A}$}) Set $u^- :=
\widetilde{u}$. Compute $\widehat{u} = P_{{\cal A}}(u^-)$ by
using \eqref{u_proj_A}.
\item[Step 4] ({\em Update}) Set $u^{k+1} := \widehat{u}$ and
$q^{k+1} := u^{k} + q^k - \widetilde{u}$\,.
\item[Step 5] ({\em Stopping criterion}) If $\|u^{k+1} -
u^k\|_{L^\infty} \le \varepsilon$, then return $\widetilde{u}$ and stop.  
Otherwise, set $k := k+1$ and go to Step 2.
\end{description}

Next is the Douglas--Rachford method from Subsection~\ref{ss:DR}.

\newpage

\noindent
{\bf Algorithm 2 (DR)}
\begin{description}
\item[Step 1] ({\em Initialization}) Choose a parameter $\lambda\in\left]0,1\right[$ and the initial iterate $u^0$ arbitrarily. 
Choose a small parameter $\varepsilon>0$, and set $k=0$. 
\item[Step 2] ({\em Projection onto ${\cal B}$})  Set $u^- = \lambda u^{k}$. 
Compute $\widetilde{u} = P_{{\cal B}}(u^-)$ by using \eqref{u_proj_B}. 
\item[Step 3] ({\em Projection onto ${\cal A}$}) Set $u^- := 2\widetilde{u}-u^k$. 
Compute $\widehat{u} = P_{{\cal A}}(u^-)$ by using \eqref{u_proj_A}.
\item[Step 4] ({\em Update}) Set $u^{k+1} := u^k + \widehat{u} - \widetilde{u}$.
\item[Step 5] ({\em Stopping criterion}) If $\|u^{k+1} - u^k\|_{L^\infty} \le \varepsilon$, then return $\widetilde{u}$ and stop.  
Otherwise, set $k := k+1$ and go to Step 2.
\end{description}

Finally, we describe the Arag\'on Artacho--Campoy algorithm from Subsection~\ref{ss:AAC}.

\noindent
{\bf Algorithm 3 (AAC)}
\begin{description}
\item[Step 1] ({\em Initialization}) Choose the initial iterate $u^0$ arbitrarily.
Choose a small parameter $\varepsilon>0$, two parameters\footnote{Arag\'on Artacho and Campoy recommend $\alpha=0.9$ and $\beta\in[0.7,0.8]$; see \cite[End of Section~7]{AAC}.} $\alpha$ and $\beta$ in $\left]0,1\right[$, and set $k=0$. 
\item[Step 2] ({\em Projection onto ${\cal B}$})  Set $u^- = u^{k}$. 
Compute $\widetilde{u} = P_{{\cal B}}(u^-)$ by using \eqref{u_proj_B}. 
\item[Step 3] ({\em Projection onto ${\cal A}$})  Set $u^- = 2\beta\widetilde{u}-u^k$.
Compute $\widehat{u} = P_{{\cal A}}(u^-)$ by using \eqref{u_proj_A}. 
\item[Step 4] ({\em Update}) 
Set $u^{k+1} := u^k +2\alpha\beta(\widehat{u}-\widetilde{u})$.
\item[Step 5] ({\em Stopping criterion}) If $\|u^{k+1} - u^k\|_{L^\infty} \le \varepsilon$, then return $\widetilde{u}$ and stop.  
Otherwise, set $k := k+1$ and go to Step 2.
\end{description}

We provide another version of each of Algorithms~1--3, as Algorithms~1b--3b, in Appendix~A.  In Algorithm~1b, we {\em monitor} the sequence of iterates which are the projections onto set ${\cal A}$, instead of monitoring the projections onto set ${\cal B}$ in Algorithm~1.  On the other hand, in Algorithms~2b--3b, the order in which the projections are done is reversed: the first projection is done onto the set ${\cal A}$ and the second projection onto ${\cal B}$.  

Although the order of projections will not matter in view of the existing results stating that convergence is achieved under any order -- see \cite[Proposition~2.5(i)]{BauMou2016}, the order does make a difference in early iterations (as well as in the number of iterations required for convergence of Algorithms~2 and 2b, as we will elaborate on later).  If our intent is to stop the algorithm early so that we can use the current iterate as an initial guess in more accurate computational optimal control algorithms, which can find the junction times with a high precision (see \cite{KayNoa1996, KayNoa2003, KayLucSim2004}), then it is desirable to implement Algorithms~1--3 above, rather than Algorithms~1b--3b, because any iterate of Algorithms~1--3 will satisfy the constraints on the control variable, while that of Algorithms~1b--3b will in general not.

\subsection{Numerical experiments}
In what follows, we study the working of Algorithms~1--3 for an instance of Problem~(P).  Suppose that the car is initially at a reference position 0 and has unit speed. It is desired that the car come back to the reference position and be at rest after one unit of time; namely that $s_0 = 0$, $s_f = 0$, $v_0 = 1$, $v_f = 0$.  For these boundary conditions, no solution exists if one takes the control variable bound $a=2.4$ or smaller but a solution does exist for $a=2.5$.  So, we use $a=2.5$.  In the ensuing discussions, we use the {\em stopping tolerance} $\varepsilon = 10^{-8}$ unless otherwise stated.

\noindent 
{\bf Discretization.}\ \ Algorithms~1--3, as well as 1b--3b, carry out iterations with functions.  For computations, we consider discrete approximations of the functions over the partition $0 = t_0 < t_1 < \ldots < t_N = 1$ such that
\[
t_{i+1} = t_i + h\,,\ \ i = 0,1,\ldots,N\,,
\]
$h := 1/N$ and $N$ is the number of subdivisions.  Let $u_i$ be an approximation of $u(t_i)$, i.e., $u_i\approx u(t_i)$, $i = 0,1,\ldots,N-1$; similarly, $x_{1,i}\approx x_1(t_i)$ and $x_{2,i}\approx x_2(t_i)$, or $x_i := (x_{1,i}, x_{2,i})\approx x(t_i)$, $i = 0,1,\ldots,N$.  In other words, the functions $u$, $x_1$ and $x_2$ are approximated by the $N$-dimensional array $u_h$, with components $u_i$, $i = 0,1,\ldots,N-1$, and the $(N+1)$-dimensional arrays $x_{1,h}$ and $x_{2,h}$, with components $x_{1,i}$ and $x_{2,i}$, $i = 0,1,\ldots,N$, respectively. We define a discretization $P^h_{\cal A}$ of the projection $P_{\cal A}$ as follows.

\begin{equation}  \label{u_proj_A_discr}
P^h_{{\cal A}}(u_h^-)(t) = u_h^- + c_1\,t_h + c_2\,,
\end{equation}
where  $t_h = (0,t_1,\ldots,t_N)$,
\begin{eqnarray}
&& c_1 = 12\,(x_{1,N} - s_f) - 6\,(x_{2,N} - v_f)\, \label{c1_sol_discr}\,, \\[1mm]
&& c_2 = -6\,(x_{1,N} - s_f) + 2\,(x_{2,N} - v_f)\, \label{c2_sol_discr}\,,
\end{eqnarray}
and $x_{1,N}$ and $x_{2,N}$ are obtained from the Euler discretization of \eqref{IVPa_ref}--\eqref{IVPb_ref}:  Given $x_{1,0} = s_0$ and $x_{2,0} = v_0$,
\begin{eqnarray}  \label{Euler}
&& x_{1,i+1} = x_{1,i} + h\,x_{2,i}\,, \label{Euler_a} \\[1mm]
&& x_{2,i+1} = x_{2,i} + h\,u_i^-(t)\,, \label{Euler_b}
\end{eqnarray}
for $i = 0,1,\ldots,N-1$.

The discretization $P^h_{\cal B}$ of the projection $P_{\cal B}$ can be defined in a straightforward manner, by simply replacing $u^-$ in \eqref{u_proj_B} with the discrete components $u_i^-$ of $u_h^-$.

\noindent 
{\bf Parametric Behaviour.}\ \ Obviously, the behaviour of Algorithms~2 and 2b, the Douglas--Rachford method, depend on the parameter $\lambda$, and the behaviour of Algorithms~3 and 3b on the two parameters $\alpha$ and $\beta$.  Figure~\ref{fig:parameters} displays the dependence of the number of iterations it takes to converge on these parameters, for various values of $a$.  The dependence for a given value of $a$ appears to be continuous, albeit the presence of downward {\em spikes}.  

The graphs for Algorithms~2 and 2b, shown in parts (a) and (c) of Figure~\ref{fig:parameters}, respectively, differ significantly from one another.  The bound $a=4$ corresponds to the case when the control constraint becomes active only at $t=0$ -- see Figure~\ref{fig:unconstr_soln}.  In other words, when $a>0$ the optimal control variable is truly unconstrained.  When $a=4$, the best value of $\lambda$ is 1 for Algorithm~2, yielding the solution in just 6 iterations.  For Algorithm~2b, the best value for $\lambda$ is 0.5, as can be seen in (c), producing the solution in 30 iterations.  Going back to Algorithm~2, with decreasing values of $a$, the values of $\lambda$ minimizing the number of iterations shift to the right.  For example, the minimum number of iterations is 91, with $a = 2.5$ and $\lambda = 0.7466$ (found by a refinement of the graph).  

As for Algorithm~2b, the minimizer for $a=2.5$ is $\lambda = 0.5982766$ and the corresponding minimum number of iterations is 38.  This is a point where a downward spike occurs and so the number of iterations is very sensitive to changes in $\lambda$.  For example, the rounded-off value of $\lambda = 0.598277$ results in 88 iterations instead of 38, and $\lambda = 0.55$ yields 444 iterations for convergence.  The number of iterations is less sensitive to the local minimizer $\lambda = 0.7608$, which results in 132 iterations.  It is interesting to note that the graph with $a=4$ appears to be an envelope for the number of iterations for all $\lambda\in]0,1[$.

\begin{figure}[t]
\begin{minipage}{80mm}
\begin{center}
\includegraphics[width=80mm]{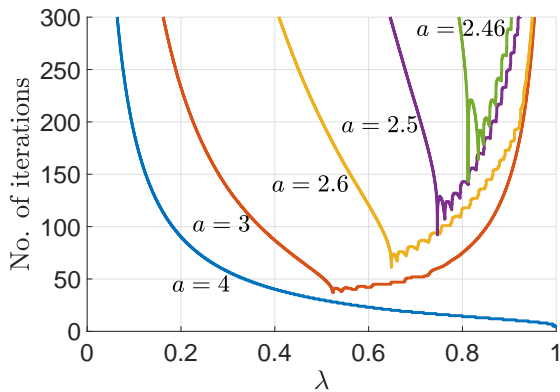} \\[3mm]
(a) Algorithm 2 (DR) \\ \
\end{center}
\end{minipage}
\begin{minipage}{80mm}
\begin{center}
\includegraphics[width=80mm]{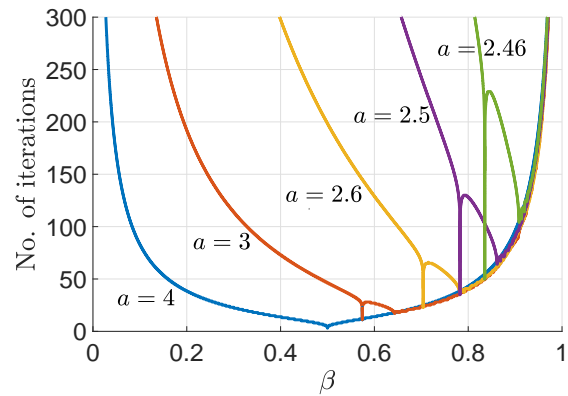} \\[3mm]
(b) Algorithms 3 (AAC) and 3b (AAC-b) \\
for $\alpha = 1$
\end{center}
\end{minipage}
\\[5mm]
\begin{minipage}{80mm}
\begin{center}
\includegraphics[width=80mm]{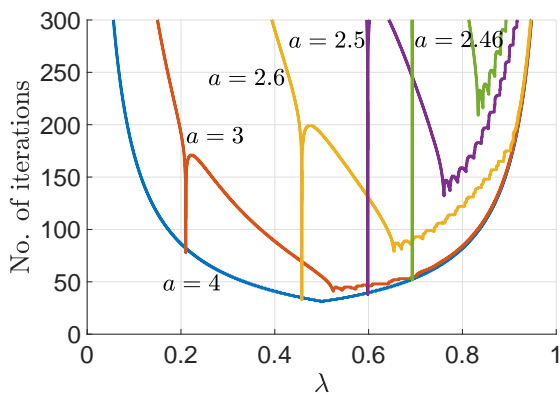} \\[3mm]
(c) Algorithm 2b (DR-b)
\end{center}
\end{minipage}
\begin{minipage}{80mm}
\begin{center}
\includegraphics[width=80mm]{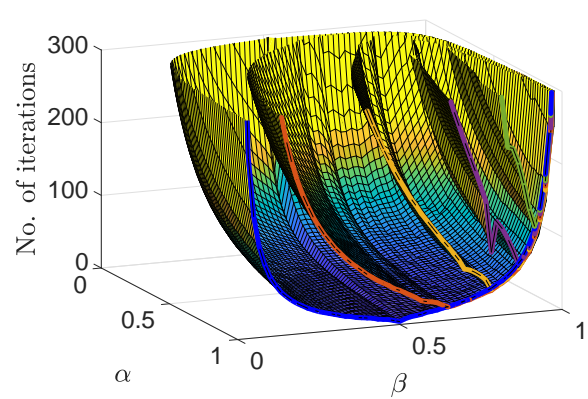} \\[3mm]
(d) Algorithms 3 (AAC) and 3b (AAC-b)
\end{center}
\end{minipage}
\
\caption{\sf Numerical experiments with $s_0 = 0$, $s_f = 0$, $v_0 = 1$, $v_f = 0$.} 
\label{fig:parameters}
\end{figure}

The graphs for Algorithms~3 and 3b, the Arag\'on Artacho--Campoy algorithm, are indistinguishable to one's eye; therefore we only display the one in Figure~\ref{fig:parameters}(b).  Part~(d) of Figure~\ref{fig:parameters} shows surface plots of the number of iterations versus the algorithmic parameters $\alpha$ and $\beta$, for the same values of $a$ as in the rest of the graphs in the figure.  It is interesting to observe that the surfaces look to be cascaded with (roughly) the outermost surface being the one corresponding to $a=4$.  The surface plot suggests that for minimum number of iterations, one must have $\alpha = 1$.  Although theory requires $\alpha < 1$, $\alpha=1$ seems to cause no concerns in this particular intance; so, we set $\alpha = 1$ for the rest of the paper.  The cross-sectional curves at $\alpha = 1$ are shown with much more precision in part~(b) of the figure.  The {\em spikes} that are observed in part~(d) can also be seen in the graph in part~(b).

In fact, the first observation one has to make here is that, for $a=4$, convergence can be achieved in merely one iteration, with $\beta = 0.5$.  This is quite remarkable, compared with Algorithms~2 and 2b.  The graphs in~(b) appear to be enveloped as well by the graph for $a=4$, as in part (c).  For the values of $a$ other than 4, the globally minimum number of iterations seems to be achieved at a downward spike, which as a result is very sensitive to changes in $\beta$.  For example, for $a=2.5$, the optimal $\beta$ value is 0.78249754 for a minimum 35 iterations.  A rounded-off $\beta = 0.782$ results in 111 iterations, and $\beta = 0.7$ yields 243 iterations.  Sensitivity at the local minimizer $\beta = 0.8617$ giving 64 iterations is far less: Choosing $\beta = 0.8$ or $0.9$ results in 128 or 90 iterations, respectively.  It is interesting to note that, as in the case of Algorithms 2 and 2b, the graphs in Figure~\ref{fig:parameters}(b) are approximately enveloped by the graph/curve drawn for $a=4$.


\begin{figure}[t]
\begin{minipage}{80mm}
\begin{center}
\includegraphics[width=80mm]{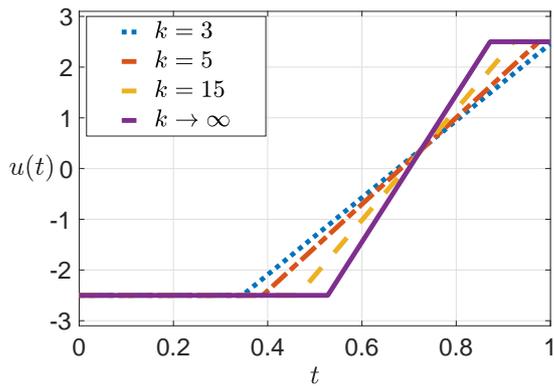} \\[3mm]
(a) Algorithm 1 (Dykstra) \\
stops after 530 iterations.
\end{center}
\end{minipage}
\begin{minipage}{80mm}
\begin{center}
\includegraphics[width=80mm]{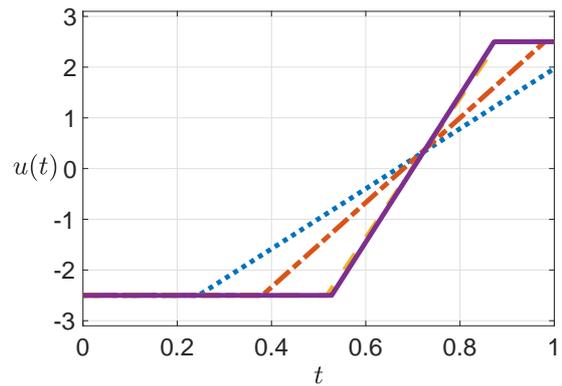} \\[3mm]
(b) Algorithm 2 (DR, $\lambda = 0.7466$) \\
stops after 91 iterations.
\end{center}
\end{minipage}
\\[5mm]
\begin{minipage}{80mm}
\begin{center}
\includegraphics[width=80mm]{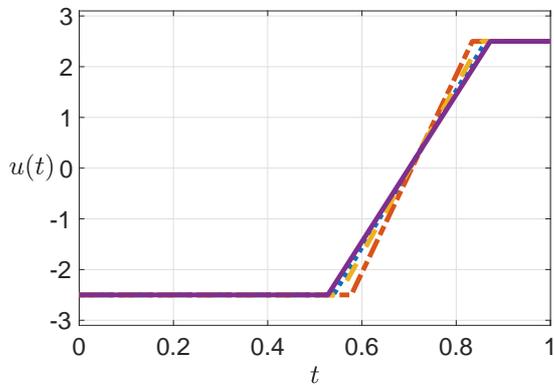} \\[3mm]
(c) Algorithm 3 (AAC, $\alpha = 1$, $\beta = 0.8617$) \\
stops after 64 iterations.
\end{center}
\end{minipage}
\begin{minipage}{80mm}
\end{minipage}
\\[5mm]
\begin{minipage}{80mm}
\begin{center}
\includegraphics[width=80mm]{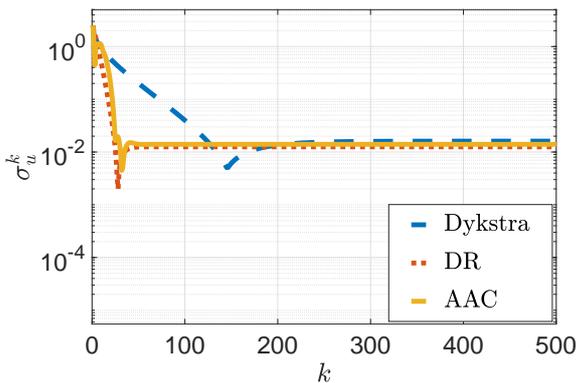} \\[3mm]
(d) $L^\infty$-error in control by Algorithms~1--3.
\end{center}
\end{minipage}
\begin{minipage}{80mm}
\begin{center}
\includegraphics[width=80mm]{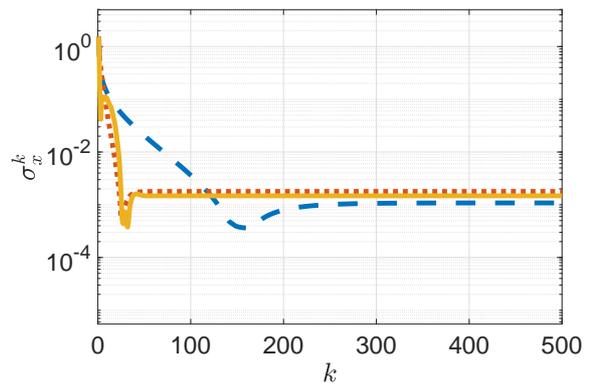} \\[3mm]
(e) $L^\infty$-error in states by Algorithms~1--3.
\end{center}
\end{minipage}
\ 
\caption{\sf Numerical experiments with $a=2.5$, $s_0 = 0$, $s_f = 0$, $v_0 = 1$, $v_f = 0$, and the number of discretization subintervals $N = 2000$.  The graphs in (a)--(c) show approximations of the optimal control function with Algorithms~1--3, after $k = 3, 5, 15$ iterations, with $\varepsilon = 10^{-8}$.  All algorithms are observed to converge to the optimal solution indicated by $k\to\infty$, in various rates.  The semi-log graphs in~(d) and (e) show the $L^\infty$ errors in the state and control variables, respectively, in each iteration of the three algorithms.} 
\label{fig:projections}
\end{figure}


\noindent 
{\bf Behaviour in Early Iterations.}\ \ Figure~\ref{fig:projections}(a)--(c) illustrates the working of all three algorithms for the same instance.  All three algorithms converge to the optimal solution, with the stopping tolerance of $\varepsilon = 10^{-8}$.  The optimal values of the algorithmic parameters, $\lambda = 0.7466$ for Algorithm~2, and $\alpha = 1$ and $\beta = 0.8617$ for Algorithm~3, have been used.  The third, fifth and fifteenth iterates, as well as the solution curve, are displayed for comparisons of behaviour.  At least for the given instance of the problem, it is fair to say from Figure~\ref{fig:projections}(c) that Algorithm~3 gets closer to the solution much more quickly than the others in the few initial iterations---see the third and fifth iterates.  It also achieves convergence in a smaller number of iterations (64 as opposed to 530 and 91 iterations of the Algorithms~1 and 2, respectively).

\noindent 
{\bf Error Analysis via Numerical Experiments.}\ \ For a fixed value of $N$, Algorithms~1--3 converge only to some approximate solution of the original Problem.  Therefore, the question as to how the algorithms behave as the time partition is refined, i.e., $N$ is increased, needs to be investigated.  For the purpose of a numerical investigation, we define, in the $k$th iteration, the following errors.  Suppose that the pair $(u^*,x^*)$ is the optimal solution of Problem~(P) and $(u_h^k,x_h^k)$ an approximate solution of Problem~(P) in the $k$th iteration of a given algorithm.  Define
\[
\sigma_u^k := \max_{0\le i\le N-1} |u^k_i - u^*(t_i)|\qquad\mbox{and}\qquad \sigma_x^k := \max_{0\le i\le N} ||x^k_i - x^*(t_i)||_\infty\,,
\]
where $||\cdot||_\infty$ is the $\ell_\infty$-norm in $\dR^2$.  For large $N$, these expressions are reminiscent of the $L^\infty$-norm, and therefore they will be referred to as the {\em $L^\infty$-error}.

For $(u^*,x^*)$ in the error expressions, we have used the discretized (approximate) solution obtained for the Euler-discretized Problem~(P) utilizing the Ipopt--AMPL suite, with $N=10^6$ and the tolerance set at $10^{-14}$.

For $N=2000$, these errors are depicted in Figure~\ref{fig:projections}(d) and (e).  From the graphs it is immediately clear that no matter how much smaller the stopping tolerance is selected, the best error that is achievable with $N=2000$ is around $10^{-2}$ for the control variable and around $10^{-3}$ for the state variable vector.  In fact, the graphs also tell that perhaps a much smaller stopping threshold than $10^{-8}$ would have achieved the same approximation to the continuous-time solution of Problem~(P).  By just looking at the graphs, one can see that Algorithm~1 could have been run just for about 300 iterations instead of 530, and Algorithms~2 and 3 could have been run for about 50 iterations to achieve the best possible approximation with $N=2000$.

In Figure~\ref{fig:errors}, we depict the same errors for $N=10^3$ (parts (a) and (b)), $N=10^4$ (in parts~(c) and (d)) and $N=10^5$ (in parts~(e) and (f)).  It is observed that, with a ten-fold increase in $N$ (which is a ten-fold decrease in $h$) the errors in both $u$ and $x$ are reduced by ten-folds, implying that the error (both in $x$ and in $u$) depends on the stepsize $h$ linearly.  This is in line with the theory of Euler-discretization of optimal control problems; see, for example, \cite{DonHag2001, DonHagMal2000}.  Furthermore, even for very large values of $N$, it can be seen from these graphs that a stopping threshold slightly smaller than $10^{-8}$ would suffice to get even more stringent error levels, such as around $10^{-4}$ for the control variable and around $10^{-5}$ for the state variable vector.  A larger stopping threshold would obviously result in smaller number of iterations.

\begin{figure}[t]
\begin{minipage}{80mm}
\begin{center}
\includegraphics[width=80mm]{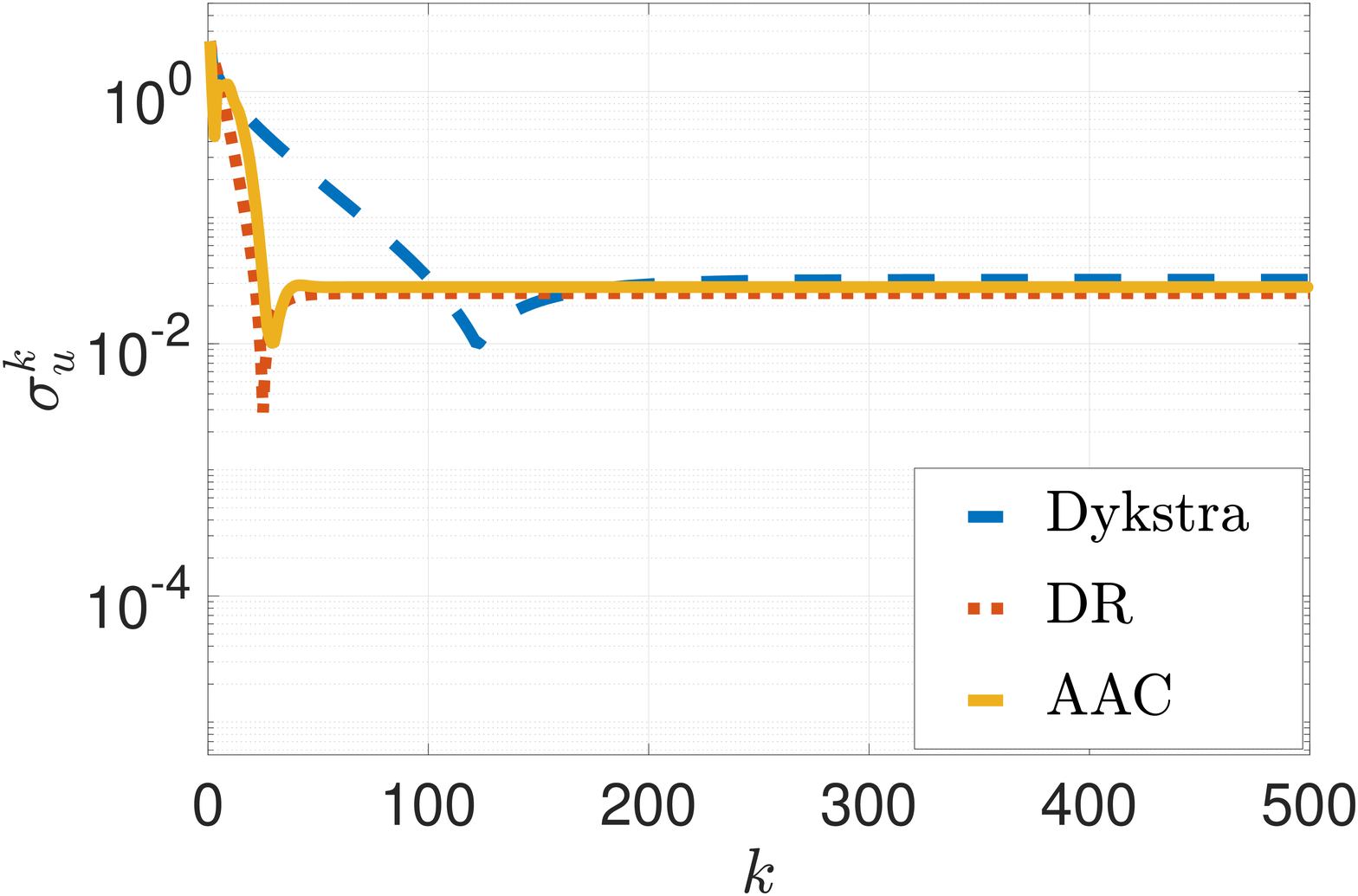} \\[3mm]
(a) $L^\infty$-error in control with $N=10^3$.
\end{center}
\end{minipage}
\begin{minipage}{80mm}
\begin{center}
\includegraphics[width=80mm]{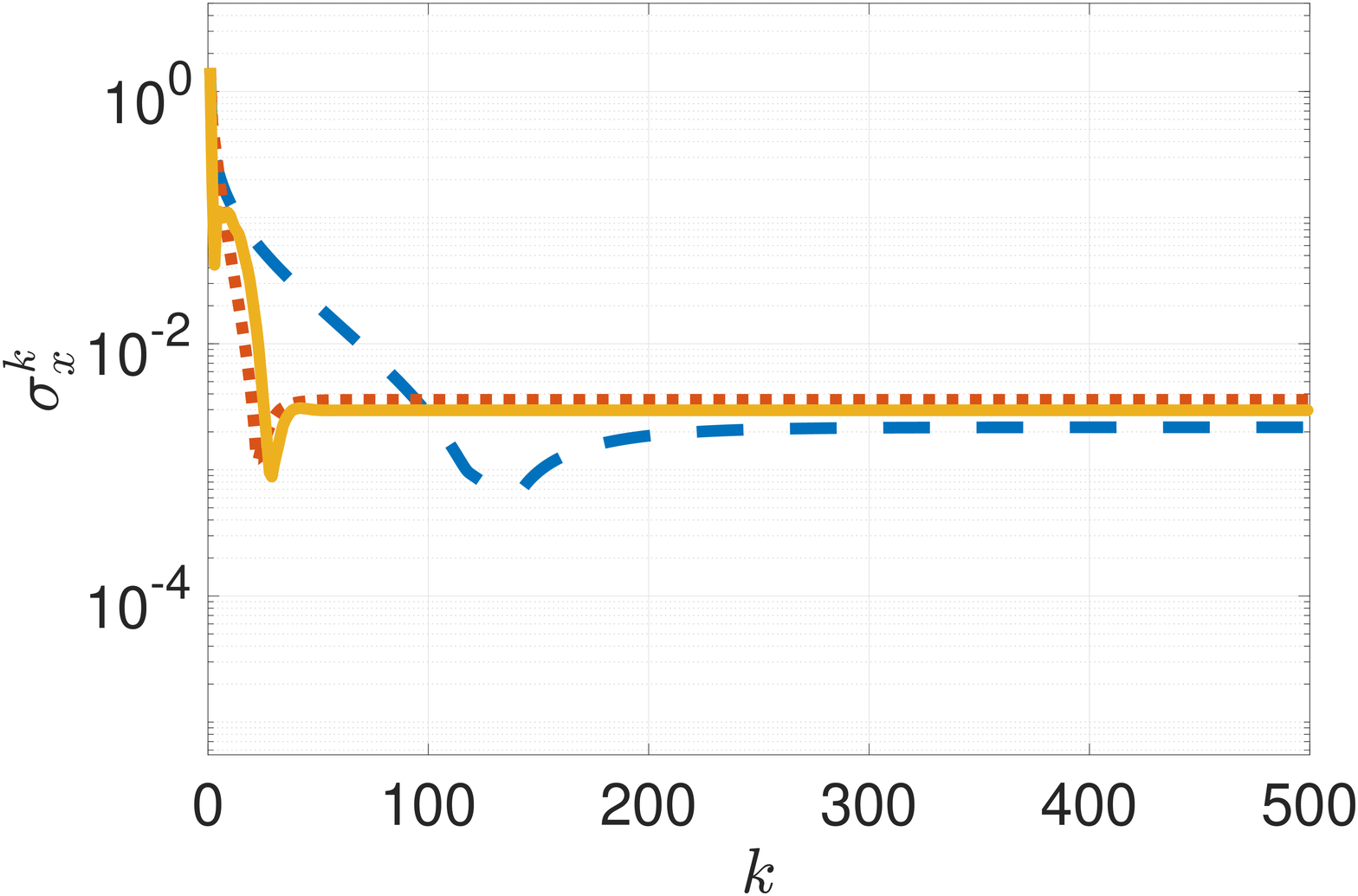} \\[3mm]
(b) $L^\infty$-error in states with $N=10^3$.
\end{center}
\end{minipage}
\\[5mm]
\begin{minipage}{80mm}
\begin{center}
\includegraphics[width=80mm]{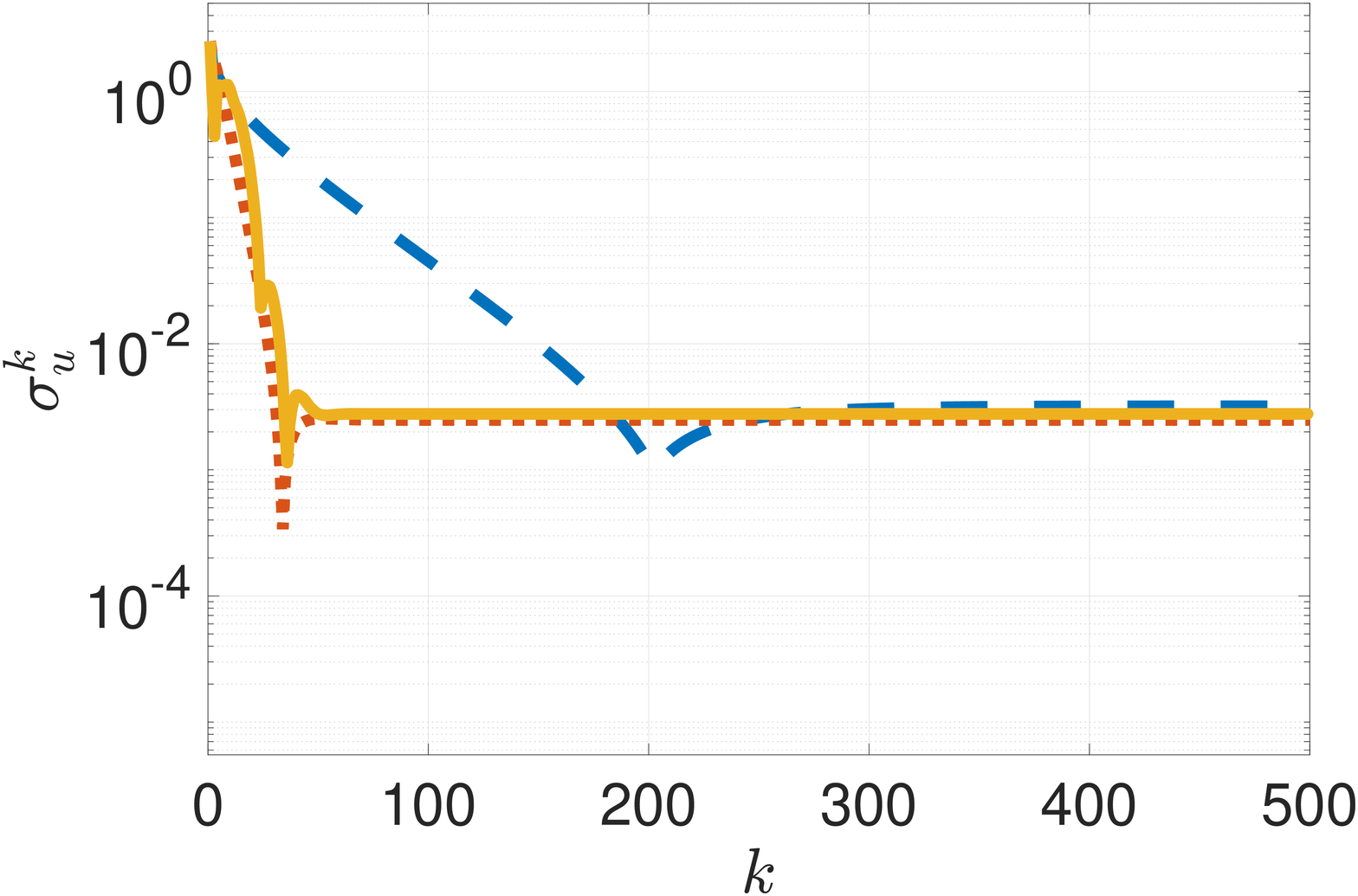} \\[3mm]
(c) $L^\infty$-error in control with $N=10^4$.
\end{center}
\end{minipage}
\begin{minipage}{80mm}
\begin{center}
\includegraphics[width=80mm]{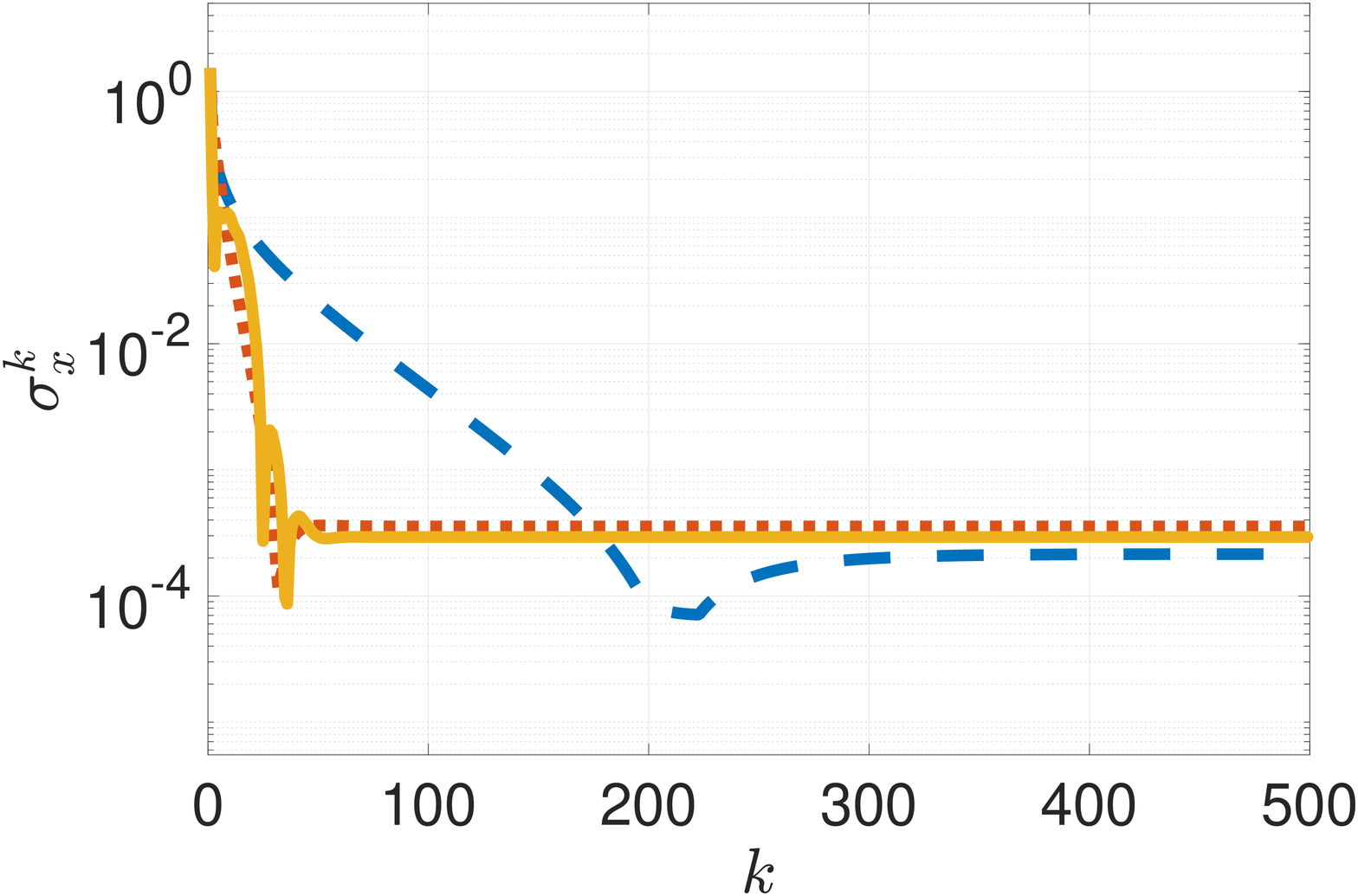} \\[3mm]
(d) $L^\infty$-error in states with $N=10^4$.
\end{center}
\end{minipage}
\\[5mm]
\begin{minipage}{80mm}
\begin{center}
\includegraphics[width=80mm]{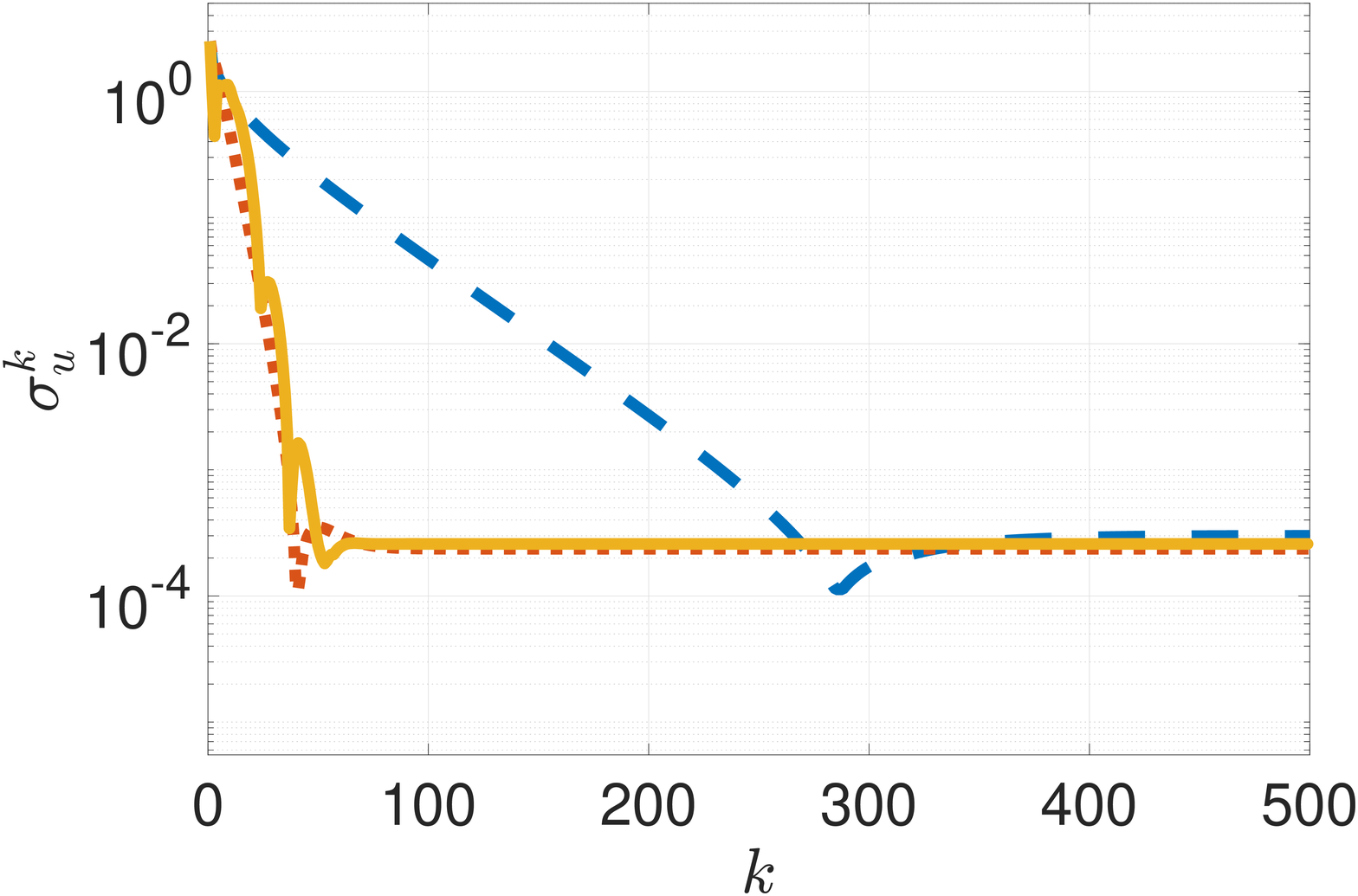} \\[3mm]
(e) $L^\infty$-error in control with $N=10^5$.
\end{center}
\end{minipage}
\begin{minipage}{80mm}
\begin{center}
\includegraphics[width=80mm]{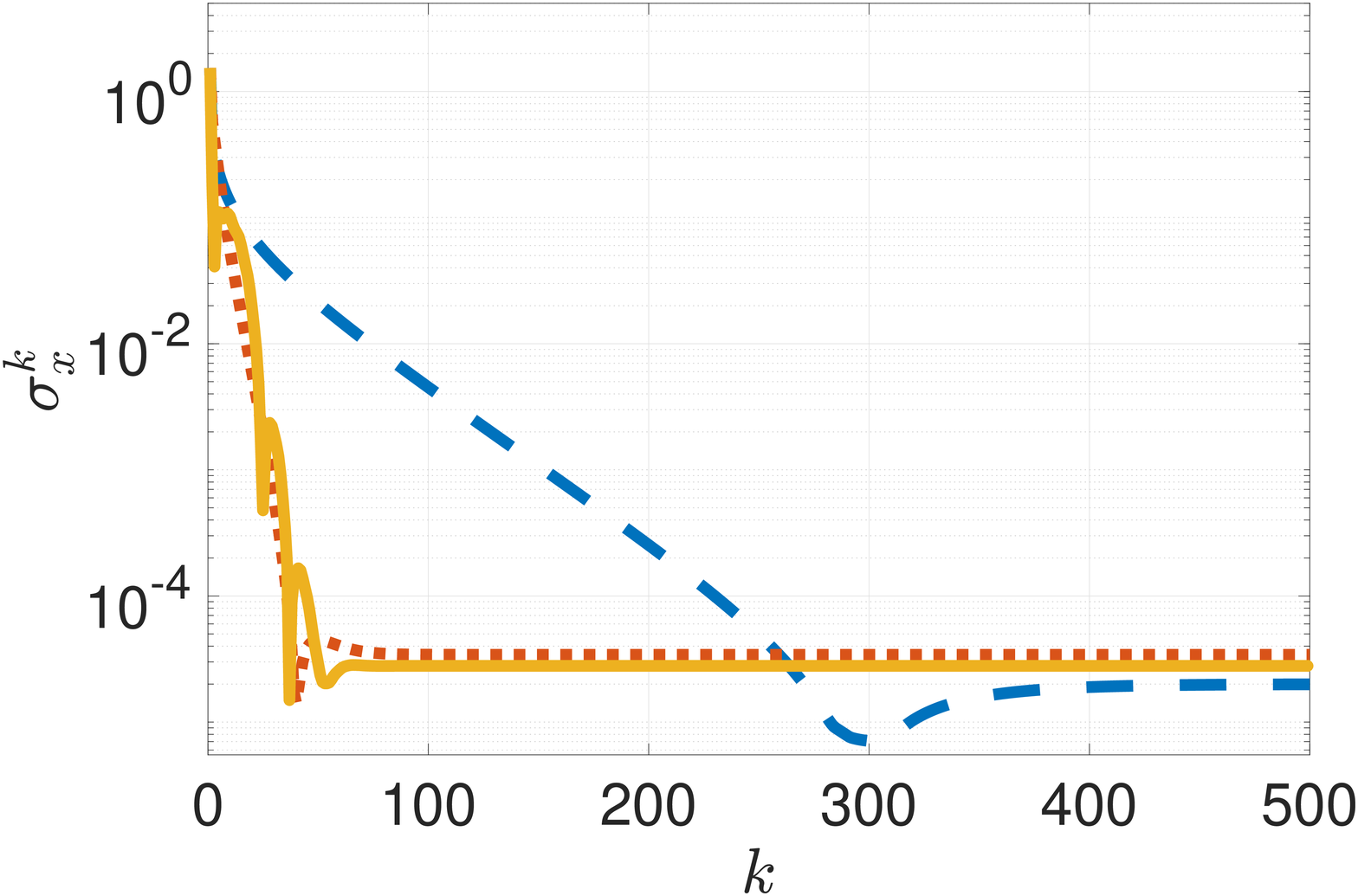} \\[3mm]
(f) $L^\infty$-error in states with $N=10^5$.
\end{center}
\end{minipage}
\ 
\caption{\sf Numerical experiments with $a=2.5$, $s_0 = 0$, $s_f = 0$, $v_0 = 1$, $v_f = 0$.  The semi-log graphs show the $L^\infty$ errors in the state and control variables, respectively, in each iteration of the three algorithms, with various $N$ from coarse ($N=1000$) to fine ($N=100000$).} 
\label{fig:errors}
\end{figure}

\clearpage

Table~\ref{table:errors} displays the values of the errors, separately in $u$ and $x$, after the stopping criteria with $\varepsilon = 10^{-8}$ was satisfied, for each of the three algorithms. A precise 10-fold reduction in error with a 10-fold increase in $N$ can be verified with these numbers, as discussed in the previous paragraph.  We have added the experiments we have carried out with Ipopt, version 3.12, an interior point optimization software \cite{WacBie2006}, which solved the direct Euler-discretization of Problem~(P), with the same values of $N$ and the same tolerance $10^{-8}$.  Ipopt, running with linear solver MA57, was paired up with the optimization modelling language AMPL~\cite{AMPL}.  The same 10-fold decreases in error cannot be observed with Ipopt, unless one sets the tolerance for Ipopt to be much smaller than $10^{-8}$, say $10^{-14}$ (which also means longer computational times). With the tolerance set at $10^{-14}$, the error values with Ipopt becomes pretty much the same as those with Dykstra (still with $\varepsilon = 10^{-8}$), which is interesting to note.

As we pointed out earlier, the same errors listed in Table~\ref{table:errors} can be achieved with bigger stopping thresholds.  For $N=10^3, 10^4, 10^5$, respectively: with $\varepsilon = 10^{-6}, 10^{-6}, 10^{-7}$, Algorithm~1 converges in 281, 359 and 454 iterations; with $\varepsilon = 10^{-5}, 10^{-5}, 10^{-7}$, Algorithm~2 converges in 65, 50 and 101 iterations; with $\varepsilon = 10^{-4}, 10^{-5}, 10^{-6}$, Algorithm~3 converges in 49, 60 and 70 iterations.  

In Table~\ref{table:CPU_times}, the CPU times (in seconds) each algorithm takes, with the respective $\varepsilon$ values listed above, are tabulated.  Note that Algorithms~1--3 have been coded and run on Matlab, 64-bit (maci64) version R2017b.  All software, including AMPL and Ipopt, were run on MacBook Pro, with operating system macOS Sierra version 10.12.6, processor 3.3 GHz Intel Core i7 and memory 6 GB 2133 MHz LPDDR3.  In Table~\ref{table:CPU_times}, the CPU times for Ipopt are listed with the tolerance $10^{-14}$, since with only this fine tolerance it is possible to obtain the same order of the error magnitudes as those obtained by Algorithms~1--3.  With $\varepsilon = 10^{-8}$, the CPU times for Ipopt are 0.06, 0.45 and 4.4 seconds, respectively, which are significantly higher than the times taken by Algorithms~1--3, in addition to worse errors.

Numerical observations suggest two joint winners: Algorithms~2 and 3, i.e., the Douglas--Rachford method and the Arag\'on Artacho--Campoy algorithm, in both accuracy and speed.

\begin{table}[t]
\begin{center}
{\footnotesize
\begin{tabular}{ScScScScSc}
$N$ & Dykstra & DR & AAC & Ipopt \\
\hline\hline
$10^3$ &  $3.2\times10^{-2}$ & $2.5\times10^{-2}$ & $2.8\times10^{-2}$ & $3.2\times10^{-2}$ \\ \hline
$10^4$ & $3.2\times10^{-3}$ & $2.5\times10^{-3}$ & $2.8\times10^{-3}$ & $7.7\times10^{-3}$ \\ \hline
$10^5$ & $3.0\times10^{-4}$ & $2.4\times10^{-4}$ & $2.6\times10^{-4}$ & $1.6\times10^{-2}$ \\ \hline
\end{tabular}
\\[5mm]
(a) $L^\infty$-error in control, $\sigma_u^k$\,. \\[5mm]
\begin{tabular}{ScScScScSc}
$N$ & Dykstra & DR & AAC & Ipopt \\
\hline\hline
$10^3$ &  $2.2\times10^{-3}$ & $3.6\times10^{-3}$ & $3.0\times10^{-3}$ & $2.2\times10^{-3}$  \\ \hline
$10^4$ & $2.1\times10^{-4}$ & $3.6\times10^{-4}$ & $2.9\times10^{-4}$ & $2.3\times10^{-4}$   \\ \hline
$10^5$ & $2.0\times10^{-5}$ & $3.4\times10^{-5}$ & $2.8\times10^{-5}$ &  $8.7\times10^{-5}$  \\ \hline
\end{tabular}
\\[5mm]
(b) $L^\infty$-error in states, $\sigma_x^k$\,.
}
\end{center}
\caption{Least errors that can be achieved by Algorithms~1--3 and Ipopt, with $\varepsilon = 10^{-8}$.}
\label{table:errors}
\end{table}

\begin{table}[h]
\begin{center}
{\footnotesize
\begin{tabular}{ScScScScSc}
$N$ & Dykstra & DR & AAC & Ipopt \\
\hline\hline
$10^3$ &  0.03 & 0.01 & 0.01 & 0.08 \\ \hline
$10^4$ & 0.16 & 0.05 & 0.05 & 0.71 \\ \hline
$10^5$ & 1.6\ \,& 0.41 & 0.28 & 7.3\ \,\\ \hline
\end{tabular}
}
\end{center}
\caption{CPU times taken by Algorithms~1--3 and Ipopt.  For $N=10^3, 10^4, 10^5$, respectively:  $\varepsilon = 10^{-6}, 10^{-6}, 10^{-7}$ for Algorithm~1, $\varepsilon = 10^{-5}, 10^{-5}, 10^{-7}$ for Algorithm~2, and $\varepsilon = 10^{-4}, 10^{-5}, 10^{-6}$ for Algorithm~3, have been used.  The tolerance for Ipopt was set as $10^{-14}$.}
\label{table:CPU_times}
\end{table}

\newpage

\section{Conclusion and Open Problems}

We have applied three well-known projection methods to solve an optimal control problem, i.e., control-constrained minimum-energy control of double integrator.  We have derived the projectors for the optimal control problem and demonstrated that they can be used in Dykstra's algorithm, the Douglas--Rachford (DR) method and the Arag\'on Artacho--Campoy (AAC) algorithm, effectively.  We carried out extensive numerical experiments for an instance of the problem and concluded that the DR and AAC algorithms (Algorithms~2 and 3) were jointly the most successful.  We also made comparisons with the standard discretization approach, only to witness the benefit of using projection methods.

It is interesting to note that when we apply alternating projections, we also seem to converge to
$P_{\mathcal{A}\cap\mathcal{B}}(0)$ even though this is not supported by existing theory.

To the best of authors' knowledge, the current paper constitutes the first of its kind which involves  projection methods and continuous-time optimal control problems. It can be considered as a prototype for future studies in this direction.  Some of the possible directions are listed as follows.
\begin{itemize}
\item The setting we have introduced could be extended to general control-constrained linear-quadratic problems.
\item  We have used some discretization of the projector as well as the associated IVP in~\eqref{u_proj_A_discr}--\eqref{Euler_b}.  This might be extended to problems in more general form.  On the other hand, for the particular problem we have dealt with in the present paper, one might take into account the fact that if $u^-(t)$ is piecewise linear then its projection is piecewise linear.  This might simplify further the expressions given in Proposition~\ref{proj_A}.
\item  Although theory for projection methods can in principle vouch convergence only for convex problems, it is well-known that the DR method can be successful for nonconvex problems, see, for example, \cite{BorSim2011}.  It would be interesting to extend the formulations in the current paper to nonconvex optimal control problems.
\item For a certain value of an algorithmic parameter, Figure~3 exhibits downward spikes.  It would be interesting to see if this phenomenon is also observed in other control-constrained optimal control problems.
\end{itemize}

\newpage

\section*{Appendix}

\noindent
{\bf Algorithm 1b (Dykstra-b)}
\begin{description}
\item[Steps 1--4] ({\em Initialization}) Do as in Steps~1--4 of Algorithm~1.
\item[Step 5] ({\em Stopping criterion}) If $\|u^{k+1} - u^k\|_{L^\infty} \le \varepsilon$, then return $u^{k+1}$ and stop.  
Otherwise, set $k := k+1$ and go to Step 2.
\end{description}

\noindent
{\bf Algorithm 2b (DR-b)}
\begin{description}
\item[Step 1] ({\em Initialization}) Choose a parameter $\lambda\in\left]0,1\right[$ and the initial iterate $u^0$ arbitrarily. 
Choose a small parameter $\varepsilon>0$, and set $k=0$. 
\item[Step 2] ({\em Projection onto ${\cal A}$})  Set $u^- = \lambda u^{k}$. 
Compute $\widetilde{u} = P_{{\cal A}}(u^-)$ by using \eqref{u_proj_A}. 
\item[Step 3] ({\em Projection onto ${\cal B}$}) Set $u^- := 2\widetilde{u}-u^k$. 
Compute $\widehat{u} = P_{{\cal B}}(u^-)$ by using \eqref{u_proj_B}.
\item[Step 4] ({\em Update}) Set $u^{k+1} := u^k + \widehat{u} - \widetilde{u}$.
\item[Step 5] ({\em Stopping criterion}) If $\|u^{k+1} - u^k\|_{L^\infty} \le \varepsilon$, then return $\widetilde{u}$ and stop.  
Otherwise, set $k := k+1$ and go to Step 2.
\end{description}

\noindent
{\bf Algorithm 3b (AAC-b)}
\begin{description}
\item[Step 1] ({\em Initialization}) Choose the initial iterate $u^0$ arbitrarily.
Choose a small parameter $\varepsilon>0$, two parameters $\alpha$ and $\beta$ in $\left]0,1\right[$, and set $k=0$. 
\item[Step 2] ({\em Projection onto ${\cal A}$})  Set $u^- = u^{k}$. 
Compute $\widetilde{u} = P_{{\cal A}}(u^-)$ by using \eqref{u_proj_A}. 
\item[Step 3] ({\em Projection onto ${\cal B}$})  Set $u^- = 2\beta\widetilde{u}-u^k$.
Compute $\widehat{u} = P_{{\cal B}}(u^-)$ by using \eqref{u_proj_B}. 
\item[Step 4] ({\em Update}) 
Set $u^{k+1} := u^k +2\alpha\beta(\widehat{u}-\widetilde{u})$.
\item[Step 5] ({\em Stopping criterion}) If $\|u^{k+1} - u^k\|_{L^\infty} \le \varepsilon$, then return $\widetilde{u}$ and stop.  
Otherwise, set $k := k+1$ and go to Step 2.
\end{description}

\end{document}